\documentclass{amsart}
\usepackage[centertags]{amsmath}
\usepackage{amsfonts}
\usepackage{amssymb}
\usepackage{amsthm}
\usepackage{newlfont}
\usepackage{color}
\usepackage{ulem}
\usepackage{geometry} 
\newtheorem{theorem}{Theorem}[section]
\newtheorem{proposition}[theorem]{Proposition}
\newtheorem{lemma}[theorem]{Lemma}
\newtheorem{corollary}[theorem]{Corollary}

\theoremstyle{definition}

\newtheorem{remark}[theorem]{Remark}
\theoremstyle{example}
\newtheorem{example}[theorem]{Example}
\theoremstyle{problem}

\begin{document}

\title[]{Functions preserving operator means}

\author[T.H. Dinh]{Trung Hoa Dinh}
\address{Department of Mathematics, Troy University, Troy, AL 36082, USA}
\email{thdinh@troy.edu}

\author[H. Osaka]{HIROYUKI OSAKA} 
\address{Department of Mathematical Sciences, Ritsumeikan
University, Kusatsu, Shiga 525-8577, Japan}
\email{osaka@se.ritsumei.ac.jp}

\author[S. Wada]{SHUHEI WADA}
\address{Department of Information and Computer Engineering,
National Institute of Technology, Kisarazu College
Kisarazu Chiba 292-0041, Japan}
\email{wada@j.kisarazu.ac.jp}

\date{\today}



\begin{abstract}
Let $\sigma$ be a non-trivial operator mean in the sense of Kubo and Ando, and let $OM_+^1$ 
the set of normalized positive operator monotone functions on $(0, \infty)$. 
In this paper, we study class of $\sigma$-subpreserving functions $f\in OM_+^1$  satisfying  
$$f(A\sigma B) \le f(A)\sigma f(B)$$
for all positive operators $A$ and $B$.
We provide some criteria for $f$ to be trivial, i.e.,   $f(t)=1$ or $f(t)=t$. 
We also establish characterizations of $\sigma$-preserving functions $f$ satisfying 
$$f(A\sigma B) = f(A)\sigma f(B)$$
for all positive operators $A$ and $B$. 
In particular, when $\lim_{t\rightarrow 0} (1\sigma t) =0$, the function $f$ preserves $\sigma$ if and only if $f$ and $1\sigma t$ are representing functions for weighted harmonic means.

\vskip 2mm

\leftline{Key words. Operator means, operator monotone functions, positive matrices, operator convexity.}

\vskip 1mm

\leftline{AMS subject classifications. 46L30, 15A45.}
\end{abstract}

\maketitle

\section{Introduction}

A real-valued function $f$ on $(0, \infty)$ is called 
operator monotone in the case when, if bounded operators 
$A$ and $B$ satisfy $0 < A \leq B$, $f(A) \leq f(B)$ holds. The two functions $f(t) = t^s$ \ $(s \in [0, 1])$  and  
$f(t) = \log t$ are well-known examples of operator monotone functions. 

In \cite{KA}, 
Kubo and Ando developed an axiomatic theory for connections and means for pairs of positive operators.
In particular, a binary operation $\sigma$ acting on a class of positive operators, $(A, B) \mapsto A\sigma B$, is 
called a connection if the following requirements are fulfilled:
\begin{enumerate}
\item[(I)] If $A \leq C$ and $B \leq D$, then $A \sigma B \leq C \sigma D$;
\item[(II)] $C(A \sigma B)C \leq (CAC) \sigma (CBC)$;
\item[(III)] If $A_n \searrow A$ and $B_n \searrow B$, then 
$A_n \sigma B_n \searrow A \sigma B$.
\end{enumerate}
Further, a mean is a connection satisfying the normalized condition:

(IV) $1 \sigma 1 = 1$.

Kubo and Ando showed that there exists an affine order-isomorphism from the class of connections 
to the class of positive operator monotone functions, which is 
given  by $\sigma \mapsto f_\sigma(t) = 1\sigma t$.

In the past few years, the theory of operator means has been intensively studied due to a vast of applications in mathematics and quantum information theory as well. Recently, Molnar and other authors obtained full description of maps preserving different types of operator means \cite{gn, molnar1, molnar2, molnar3}. For example, in  \cite{molnar1}  Molnar studied maps $\Phi: B(H)^+ \to B(H)^+$ satisfying the condition
$$
\Phi(A \sharp B) = \Phi(A) \sharp \Phi(B),
$$
where $B(H)^+$ is the set of positive invertible operators acting on some Hilbert space $H$ and $\sharp$ is the well-known geometric mean  $A\sharp B = A^{1/2}(A^{-1/2}BA^{-1/2})^{1/2}A^{1/2}$. In a consequent paper \cite{molnar2} he studied the same problem for general operator mean $\sigma$,
\begin{equation}\label{mean}
    \Phi(A  \sigma B) = \Phi(A) \sigma \Phi(B).
\end{equation}
In a recent paper \cite{gn}, Gaal and Nagy considered preserver problems related to quasi-arithmetic means of invertible positive operators. We would like to emphasise that most of works concerned maps from  $B(H)^+$ to  $B(H)^+$. Therefore, it is natural to ask the following question:

{\bf Question 1.} {\it For a fixed operator mean $\sigma$, what is the class of real-valued function $f$ satisfying 
\begin{equation}\label{33}
f(A \sigma B) = f(A) \sigma f(B)
\end{equation}
for any positive definite matrices $A$ and $B$}?

Notice that Fujii and Nakamura also considered related inequalities to (\ref{33}) a little earlier in \cite{FN}. For a positive operator monotone function $f$ on $(0, \infty)$ and an operator mean $\sigma$, they studied the following inequalities
$$
f(A \sigma B) \leq f(A) \sigma f(B) \quad \quad  (*_L)$$
and 
$$
f(A \sigma B) \geq f(A) \sigma f(B), \quad  \quad  (*_R)
$$
whenever $A$, $B$ are positive definite matrices. A function $f$ satisfying the inequality $(*_L)$ for any positive definite matrices $A$ and $B$ is called {\it $\sigma$-subpreserving}. Similarly, using inequality $(*_L)$ we may define the class of {\it$\sigma$-superpreserving functions}. A function $f$ satisfying both $(*_L)$ and $(*_R)$ is said to be {\it $\sigma$-preserving}. This definition is similar to one in (\ref{mean}) but for real-valued functions. Recall that in \cite{FN} Fujii and Nakamura showed that a non-trivial operator mean $\sigma$ is the weighted harmonic mean if and only if $\sigma$ satisfies one of $(*_L)$ or $(*_R)$ for all positive operator monotone functions $f$ and for any positive operators $A$ and $B$. In other words, they established a characterization of means satisfying one of $(*_L)$ or $(*_R)$. Therefore, it is also natural to ask the following question:

{\bf Question 2.} For a fixed operator mean $\sigma$ what is the class of operator monotone functions satisfying one of $(*_L)$ or $(*_R)$?

It is worth noting that  $\sigma$ is the arithmetic mean the class of $\sigma$-subpreserving functions coincides with the class of well-known operator convex functions. In \cite{AH01} Hiai and Ando obtained a full characterization of operator log-convex functions $f$ satisfying
$$
f\left(\frac{A +B}{2}\right) \le f(A) \sharp (B).
$$
In \cite{hoa} the first author and his co-author defined the class of operator $\sigma \tau$-convex functions for operator means $\sigma$ and $\tau$ by the inequality
$$
f(A\sigma B) \le f(A) \tau f(B)
$$
For $\tau = \sigma$ this is nothing but the class of $\sigma$-subpreserving functions.

The main aim of this paper is to study Questions 1 and 2. In the next section for a fixed operator mean $\sigma$ we study $\sigma$-subpreserving functions. More precisely, We  provide some conditions for a $\sigma$-subpreserving function $f$ to be trivial, i.e., a constant or identity. These types of functions are similar to the standard maps on $B(H)^+$ satisfying (\ref{mean}). In Section 3, we establish a full characterization of $\sigma$-preserving functions. Finally, we obtain a refinement of Fujii and Nakamura's result mentioned above.

\section{The class of $\sigma$-subpreserving functions}\label{section name subpreserving}

Let $\sigma$ be a fixed operator mean and let $OM_+^1$ the class of all normalized positive operator monotone functions on $(0, \infty)$. We define the class $OM_+^1(\sigma) \subset OM_+^1$ as follows:
$$OM_+^1(\sigma):=\{f\in OM_+^1~|~f(A\sigma B)\le f(A)\sigma f(B)\quad \text{for all } A,B> 0\}. $$
Since the arithmetic mean is the biggest mean, it is obvious that for any $\alpha \in [0, 1],$ $$OM_+^1(\sigma)\supseteq OM_+^1(\nabla_\alpha).$$
Notice that for each $f\in OM_+^1$, 
$f$ can be uniquely extended to a continuous function on $[0,\infty)$ defined by 
$f(0):=\lim_{t\rightarrow 0+} f(t)$. 

This section focuses on the class $OM_+^1(\sigma)$ of $\sigma$-subpreserving functions. Firstly, we obtain some properties of $OM_+^1(\sigma)$. And then, we provide some conditions for a function $f\in OM_+^1(\sigma)$ to be trivial.



For an operator monotone function $f$ on $(0, \infty)$, it is well-known that $f(t) \ge t$ on $(0, 1)$ and $f(t) \le t$ on $(1, \infty).$ The following lemma is a straightforward consequence of the concavity of operator monotone functions. 

\begin{lemma}\label{lemma:constant}
Let $f \in OM_+^1$. Suppose that there is a number $t_0 \in (0, \infty)$ \ $(t_0 \not= 1)$
such that $f(t_0) = 1$. Then $f(t) = 1$ on $(0,\infty)$. 
\end{lemma}



Recall that if $f:(0, \infty) \to (0, \infty)$ is operator monotone,
then its transpose $f'(t) = tf(t^{-1})$, adjoint $f^*(t) = f(t^{-1})^{-1}$, and dual $f^\perp(t) = tf(t)^{-1}$
are also operator monotone \cite{KA}. Furthermore, $f$ is said to be symmetric if $f = f'$ and $f$ is self-adjoint if $f = f^*$.
It was showed in \cite{FJ92} that if $f$ satisfies $\dfrac{df}{dt} \big |_{t=1} =\lambda$ and $f(1) = 1$, then 
$\lambda \in [0, 1]$ and the corresponding operator mean $\sigma_f$ lies between the weighted harmonic mean and the weighted arithmetic mean. 

\begin{lemma}\label{prp:condition of 0}
Let $f$ and $\Phi$ be functions in $OM_+^1\backslash \{1, t\}$ such that $f\in OM_+^1(\sigma_\Phi)$ and $\Phi^*(0)=0$. Then, $f^*(0)=0.$
\end{lemma}
\begin{proof}
Assume, to the contrary, that $f^*(0) > 0$. 
Then we have
\begin{equation}\label{343}
\Phi^*(f^*(0))=1\sigma_{\Phi^*}f^*(0)=f^*(1)\sigma_{\Phi^*}f^*(0)
\le
f^*(1\sigma_{\Phi^*} 0)=
f^*(\Phi^*(0))=
f^*(0).
\end{equation}
Since $0<f^*(0)\le f^*(1)=1$ and $\Phi^*\not\in \{1,t\}$, from (\ref{343}) it  implies that $\Phi^*(f^*(0)) = f^*(0)$. Consequently, $f^*(0)=1 = f^*(1)$, hence $f^*(t)=1$ which is a contradiction.
\end{proof}
\begin{remark}\label{prp:condition of 0 at Phi}

Using a similar argument in the proof of Lemma \ref{prp:condition of 0} one can see that for $f$ and $\Phi \in OM_+^1\backslash \{1, t\}$ and $f \in OM_+^1(\sigma_\Phi)$, if $f(0) = 0$, then  $\Phi(0) = 0$. 
\end{remark}

\begin{lemma}\label{key lemma1}
Let $f, \Phi \in OM_+^1$ such that $\Phi(0)=f(0) = 0$. Let 
$$
P=\begin{pmatrix}
1 & 0\\
0 & 0 \\
 \end{pmatrix}
,\ 
Q= \begin{pmatrix}
x + y & x - y \\
x - y & x + y\\
 \end{pmatrix}, \quad x, y >0. $$
Then,
$$f(P \sigma_{\Phi'}Q)=f\left(
\Phi'\left(
\frac{4xy}{x + y}\right)
\right)
P,
$$
and 
$$f(P) \sigma_{\Phi'}f (Q) = \Phi'\left(\frac{2f(2x)f(2y)}{f(2x) + f(2y)}\right)
P,
$$
where $\Phi'$ is the transpose of $\Phi$. 
\end{lemma}

\begin{proof} We prove the first identity, the second one can be obtained similarly. Let
$$
U= \frac{1}{\sqrt{2}} \begin{pmatrix}
1 & 1\\
1 & -1 \\
 \end{pmatrix}.
$$ It is obvious that $U$ is Hermitian and unitary matrix. Then we have
\begin{align*}
P \sigma_{\Phi'} Q 
&= UU(P \sigma_{\Phi'} Q)UU\\
&= U(UPU \sigma_{\Phi'} UQU)U\\
&= U\left\{
\left(
\frac{1}{2}
\begin{pmatrix}
1 & 1\\
1 & 1 \\
 \end{pmatrix}
\right)
\sigma_{\Phi'}
\begin{pmatrix}
2y & 0\\
0 & 2x \\
 \end{pmatrix}
\right\}U
\\
&= 
U\left\{
\begin{pmatrix}
2y & 0\\
0 & 2x \\
 \end{pmatrix}
\sigma_{\Phi}
\left( 
\frac{1}{2}\begin{pmatrix}
1 & 1\\
1 & 1 \\
 \end{pmatrix}
\right)
\right\}U
\\
&= U
\begin{pmatrix}
2y & 0\\
0 & 2x \\
 \end{pmatrix}^{1\over 2}
\Phi\left(
\begin{pmatrix}
2y & 0\\
0 & 2x \\
 \end{pmatrix}^{-1\over 2}
\left(
\frac{1}{2}\begin{pmatrix}
1 & 1\\
1 & 1 \\
 \end{pmatrix}
\right)
\begin{pmatrix}
2y & 0\\
0 & 2x \\
 \end{pmatrix}^{-1 \over 2}
\right)
\begin{pmatrix}
2y & 0\\
0 & 2x \\
 \end{pmatrix}^{1\over 2}
U
\\
&= U
\begin{pmatrix}
\sqrt{2y} & 0\\
0 & \sqrt{2x} \\
 \end{pmatrix}
\Phi\left(\frac{1}{4}
\begin{pmatrix}
\frac{1}{y}&\frac{1}{\sqrt{xy}}\\
\frac{1}{\sqrt{xy}}&\frac{1}{x}
\end{pmatrix}
\right)
\begin{pmatrix}
\sqrt{2y}&0\\
0&\sqrt{2x}\\
\end{pmatrix}
U\\
&=
 U
\begin{pmatrix}
\sqrt{2y} & 0\\
0 & \sqrt{2x} \\
 \end{pmatrix}
\Phi \left( {1\over 4}
{{x+y}\over {xy}}
{{xy}\over {x+y}}
\begin{pmatrix}
{1\over {y}} & {1\over {\sqrt{xy}}} \\
{1\over {\sqrt{xy}}} & {1\over {x}} 
\end{pmatrix}
\right)
\begin{pmatrix}
\sqrt{2y}&0\\
0&\sqrt{2x}\\
\end{pmatrix}
U
\\
&=
{1\over 4}
 U
\begin{pmatrix}
\sqrt{2y} & 0\\
0 & \sqrt{2x} \\
 \end{pmatrix}\left(
\Phi \left({1\over 4} {{x+y}\over {xy}}
\right)
{{4xy}\over {x+y}}
\begin{pmatrix}
{1\over {y}} & {1\over {\sqrt{xy}}} \\
{1\over {\sqrt{xy}}} & {1\over {x}} 
\end{pmatrix} 
\right)
\begin{pmatrix}
\sqrt{2y}&0\\
0&\sqrt{2x}\\
\end{pmatrix}
U\\
&=
\Phi' \left(
{{4xy}\over {x+y}}
\right)
\begin{pmatrix}
1&0\\
0&0\\
\end{pmatrix}.
\end{align*}
Consequently, 
$$
f(P \sigma_{\Phi'} Q) 
= f\left(\Phi'\left(\frac{4xy}{x + y}\right)\right)
P.
$$

\end{proof}

Now we are ready to prove the main result of this section.   
\begin{theorem}\label{prp:operator inequality}
Let $f \in OM_+^{1}$ and $\Phi\in OM_+^1\backslash \{1,t\}$ such that $\Phi^*(0) = 0$. If $f \in OM_+^{1}(\sigma_{\Phi})$ and $f(0)=0$, then $f(t) = t$ for all $t \in [0, \infty)$.
\end{theorem}

\begin{proof}
Firstly, notice that from the definition of the adjoint means, we have 
\begin{equation}\label{434}
    A^{-1} \sigma_\Phi B^{-1} = A \sigma_\Phi^*B.
\end{equation}
Since $f \in OM_+^{1}(\sigma_{\Phi})$, for any positive definite matrices $A$ and $B$ we have
$$f(A^{-1}\sigma_\Phi B^{-1})\le f(A^{-1}) \sigma_\Phi f(B^{-1}).$$ Using the property (\ref{434}) one can see that the last inequality is equivalent to the following
\begin{equation}\label{ineq222}
    f^*(A\sigma_{\Phi^*} B)\ge f^*(A) \sigma_{\Phi^*} f^*(B).
\end{equation}
By a standard limit process and the continuity of $f^*$ and $\Phi^*$, it is obvious that (\ref{ineq222}) is still true for positive semidefinite matrices $A$ and $B$. On the other hand, from the assumption that $f \in OM_+^{1}(\sigma_{\Phi})$ and $\Phi^*(0) = 0$, by Lemma~\ref{prp:condition of 0}, it follows that $f^*(0) = 0$. Applying Lemma~\ref{key lemma1} for the functions $f^*$ and $\Phi^*$, on account of (\ref{ineq222}) we get 
\begin{equation}\label{545}
f^*\left({\Phi^*}'\left(\frac{4xy}{x+y}\right)\right) \geq {\Phi^*}'\left(\frac{2f^*(2x)f^*(2y)}{f^*(2x) + f^*(2y)}\right).
\end{equation}
Since $f(0)=0$, 
$$\lim_{y \to \infty}f^*(2y) = \lim_{y \to \infty}f^{-1}((2y)^{-1})=\infty.$$ Then, tending $y$ to $\infty$, from (\ref{545}) we obtain $$f^*({\Phi^*}'(4x)) \geq {\Phi^*}'(2f^*(2x)).$$ Consequently, for  $x = \frac{1}{4}$, 
$$1 = f^*(1) \ge {\Phi^*}'\left(2f^*\left(\frac{1}{2}\right)\right)
= {\Phi^*}'({f^*}'(2)) \geq {\Phi^*}'({f^*}'(1)) = 1.$$
Therefore, ${\Phi^*}'({f^*}'(2)) = 1$. Since ${\Phi^*}'$ is not a constant, it implies that ${f^*}'(2) = 1$, and hence, 
 ${f^*}'(t)=1$ for all $t > 0$. Thus, $f(t) = t$ for all $t \ge 0.$
\end{proof}

The Ando-Hiai log-majorization theorem \cite{ando} states that for $A, B \ge 0$ and $r \ge 1$,
$$
A^r \sharp_\alpha B^r\preceq_{\log} (A^r \sharp_\alpha B^r).
$$
It turns out that we can only compare matrices $(A\# B )^r $ and $A^r\#B^r$ with respect to the Loewner order for some special values of $r$. 
\begin{corollary}
Let $r\in {\mathbb R}$ and $\alpha\in (0,1)$. If 
$$(A\#_\alpha  B )^r \le A^r\#_\alpha B^r$$
for all $A,B>0$, then $r \in\{-1, 0, 1\}$. 
\end{corollary}
\begin{proof}
Suppose $r\not=0$. For $A, B > 0$, applying the assumption for $A^{-1}$ and $B^{-1}$ we get
$$ 
(A^{-1}\sharp_\alpha B^{-1})^r \leq (A^{-1})^r \sharp_\alpha (B^{-1})^r,
$$
or,
$$
\left\{(A^{-1})^r\sharp_\alpha (B^{-1})^r\right\}^{-1} \leq \left\{(A^{-1} \sharp_\alpha B^{-1})^r\right\}^{-1}.
$$
Consequently,
$$
A^r \sharp_\alpha B^r \leq (A\sharp_\alpha B)^r.
$$
Therefore, from the assumption and the last inequality it implies that  
$$(A\#_\alpha  B )^r = A^r\#_\alpha B^r,  \quad (A\#_\alpha  B )^{-r} = A^{-r}\#_\alpha B^{-r} \quad \text{ and } \quad  
(A\#_\alpha  B )^{1\over r} = A^{1\over r}\#_\alpha B^{1\over r}.$$
Thus, it is sufficient to show the case if $r\in (0,1]$ . By Theorem \ref{prp:operator inequality},  
we have $r=1$.
\end{proof}

\begin{example}

For $p \in [-1, 2]$ the Petz-Hasegawa function   
$PH_p$  \cite{HP} is defined as $$PH_p(t) = p(1-p)\frac{(t-1)^2}{(t^p - 1)(t^{1-p}-1)}.$$ 
It is obvious that $PH_p^*(0)=0$ for all $p\in [-1,2]$. Hence, by Theorem ~\ref{prp:operator inequality}, if $f\in OM_+^1(\sigma_{PH_p})$ and $f(0)=0$, then $f(t)=t$.
\end{example}

\vskip 3mm

\vskip 3mm



As a consequence of Theorem \ref{prp:operator inequality} we establish a condition on the dual of $\Phi$ so that the function $f$ in $OM_+^1(\sigma_\Phi)$ is trivial.

\begin{corollary}\label{prp:characterization of trivial}
Let $\Phi\in OM_+^1\backslash\{1,t\}$ with $$\Phi(0)+\Phi'(0)>0$$. 
If $f\in OM_+^1(\sigma_\Phi)$  and $f(0)=0$, 
then $f(t) = t$. 
\end{corollary}

\begin{proof}
From the assumption we get $\Phi'(0)^{-1} >0$, or, equivalently, 
$$
\lim_{t \to 0} \frac{\Phi^*(t)}{t} > 0.
$$
Consequently, $\Phi^*(0)=\lim_{t \to 0} \Phi^*(t) =0.$ By Theorem \ref{prp:operator inequality}, $f(t) = t$ for all $t \in [0, \infty).$
\end{proof}

\begin{example}
The Stolarsky means \cite{ST} are defined as  
$$ S_\alpha(s,t):=\left( {{s^\alpha -t^\alpha} \over {\alpha(s-t)}}\right)^{1 \over {\alpha-1}} 
\text{ for } \alpha \in [-2, 2] \backslash \{0,1\}.$$ 
Denote by $S_0(s,t):= \lim_{\alpha \rightarrow 0 } S_\alpha (s,t)$ and 
$S_1(s,t):= \lim_{\alpha \rightarrow 1 } S_\alpha(s,t)$. 
It is known \cite{N} that $S_\alpha(1,t) \in OM_+^1$ for all 
$\alpha\in [-2,2]$  and  
$$S_2=f_\nabla,  \quad S_{-1}=f_\# \text{ and } S_0=f_\lambda, $$
where $\lambda$ is the logarithmic mean defined by $f_\lambda(t)={{t-1}\over {\log t}}$.   
A simple calculation shows that the function $\Phi(t) := S_\alpha(1,t)$ satisfies conditions in either Theorem \ref{prp:operator inequality} or Corollary \ref{prp:characterization of trivial}. Let $f\in OM_+^1(\sigma_\Phi)$ and $f(0)=0$. For any $\alpha\in [-2,0]$, $\Phi^*(0)=0$. Then by Theorem \ref{prp:operator inequality}, $f(t)=t$. When $\alpha\in (0,2]$, since $\Phi(0)+\Phi'(0)>0$,
by Corollary \ref{prp:characterization of trivial}, we have $f(t)=t$.  Thus, the class of  functions $f$ in $OM_+^1(\sigma_\Phi)$ such that $f(0)=0$ is trivial. 
\end{example}


\vskip 3mm


\remark{
Considering the above-mentioned arguments, we have provided sufficient conditions for a function $f\in OM_+^1(\sigma)$ to be trivial. However, these conditions are not complementary. 
Indeed,
if $\alpha\in (0,1)$ and $\Phi=!_\alpha$, then 
$$\Phi^*(0)>0 \text{ and } \Phi(0)+\Phi'(0)=0$$
and our results are inconclusive for $\Phi$. 
}
\bigskip

As stated in \cite{A} and \cite{FN} the relation $OM_+^1(!_\alpha)=OM_+^1$ holds for  $\alpha\in [0,1]$. In the following result, we establish some conditions for the identity $OM_+^1(\sigma) =OM_+^1$ happens. 
\begin{proposition}\label{harmonic1}
Let $\sigma$ be an operator mean and $\alpha\in (0,1)$. Then, 
the following statements are equivalent:
\begin{itemize}
\item[(I)] $OM_+^1(\sigma)\ni f_{!_\alpha}$;
\item[(II)] $\sigma=!_\beta$ for some $\beta\in [0,1]$;
\item[(III)] $OM_+^1(\sigma)=OM_+^1$.
\end{itemize}
\end{proposition}
\begin{proof}
It is sufficient to show (I) $\Rightarrow$ (II). Let $f(t):=1!_\alpha t$. Then, we have
\begin{align*}
f^*(A\sigma^* B)&=f((A\sigma^* B)^{-1})^{-1} \\
&=
f(A^{-1}\sigma B^{-1})^{-1} \\
&\ge 
\left( f(A^{-1})\sigma f(B^{-1}) \right)^{-1} \\
&=
\left( f^*(A)^{-1}\sigma f^*(B)^{-1} \right)^{-1} \\
&= 
f^*(A)\sigma^*f^*(B). 
\end{align*}
Consequently,
$$(1-\alpha) + \alpha A\sigma^* B 
\ge ((1-\alpha) + \alpha A )
 \sigma^*
((1-\alpha) + \alpha B ).$$
Furthermore, the concavity of an operator mean implies 
\begin{align*}
((1-\alpha) + \alpha A )
 \sigma^*
((1-\alpha) + \alpha B )
&\ge 
(1-\alpha)\sigma^* (1-\alpha) +(\alpha A)\sigma^* (\alpha B) \\
&= (1-\alpha) + \alpha A\sigma^* B.
\end{align*}
Therefore,
$$(1-\alpha) + \alpha A\sigma^* B  = ((1-\alpha) + \alpha A )\sigma^*((1-\alpha) + \alpha B ).$$
Consequently, for $A=1$ and $B=t$,  
$$(1-\alpha) + \alpha \phi(t) = \phi( (1-\alpha) + \alpha t ),$$
where $\phi(t)=1\sigma^* t$. 
Differentiating both sides, we have
$$\alpha {{d\phi}\over {dt}}(t)=\alpha {{d\phi}\over {dt}}( (1-\alpha) + \alpha t ).$$
Thus, ${{d\phi}\over {dt}}$ is constant and $\phi(t)=(1-\beta)+\beta t$ for some $\beta \in [0,1]$. 
\end{proof}




\section{Class of $\sigma$-preserving functions}

In this section, for each operator mean $\sigma$, we study the 
class of functions $f$ preserving $\sigma$,
\begin{equation}\label{equation1}
 f(A\sigma B)=f(A)\sigma f(B) \quad \text{ for all }\quad  A,B>0.
\end{equation}

\begin{lemma}\label{equation nabla}
Let $f, \Phi \in OM_+^1$ with $\Phi(0) = 0$ and $f(0) = 0$.
If $f$ preserves $\sigma_{\Phi}$, then 
$$f^*\left(\Phi'^*\left(x~\nabla~y \right)\right) =\Phi'^*\left( f^*(x)~ \nabla~  f^*(y) \right)$$
for all real numbers $x,y>0$.
\end{lemma}
\begin{proof}
By Lemma \ref{key lemma1}, for all real numbers $x,y>0$, we have 
$$f\left(\Phi'\left(x~!~y \right)\right) =\Phi'\left( f(x)~ !~  f(y) \right).$$
Consequently, 
\begin{align*}
f^*\left(\Phi'^*\left(x~\nabla~y \right)\right) 
&=
f\left(\Phi'\left((x~\nabla~y)^{-1} \right)\right)^{-1} \\
&=
f\left(\Phi'\left((x^{-1}~!~y^{-1}) \right)\right)^{-1} \\
&=
\Phi'\left( f(x^{-1})~ !~  f(y^{-1}) \right)^{-1}\\
&=
\Phi'^*\left( 
\left(f(x^{-1})~ !~  f(y^{-1}) \right)^{-1}
\right)\\
&=
\Phi'^*\left( 
\left(f^*(x)^{-1}~ !~  f^*(y)^{-1} \right)^{-1}
\right)\\
&=
\Phi'^*\left( 
f^*(x) \nabla  f^*(y) 
\right).
\end{align*}

\end{proof}

\begin{theorem}\label{equation main theorem}
Let $f, \Phi \in OM_+^1\backslash\{1,t\}$ with $\Phi(0) = 0$. 
Then, the following statements are equivalent:
\begin{itemize}
\item[(I)] $f$ preserves $\sigma_{\Phi}$; 
\item[(II)] $\Phi=f_{!_\alpha}$  and $f=f_{!_\beta}$ for some $\alpha, \beta \in (0,1)$.
\end{itemize}
\end{theorem}

\begin{proof} 
The implication (II) $\Rightarrow $ (I) is obvious. We show (I) $\Rightarrow $ (II). 
First, notice that $\Phi'^*\circ f^*=f^*\circ \Phi'^*$. Indeed,
$$f(\Phi(t))=f(1\sigma_\Phi t)=f(1)\sigma_\Phi f(t)=\Phi(f(t)).$$
Now assume that $f(0)\not=0$. Then
$$f(0)=\lim_{t\rightarrow 0}f(1\sigma_\Phi t)=  
\lim_{t\rightarrow 0} f(1)\sigma_\Phi f(t)=\lim_{t\rightarrow 0}\Phi(f(t))=\Phi(f(0))>0$$
and 
$$1 = { 1 \over {f(0)}} \Phi(f(0))= \Phi'\left( { 1 \over {f(0)}}\right) .$$ 
Since $\Phi'(t)\not=1$, it follows that   
$f(0)=1$, which contradicts the fact that $f$ is non-trivial. Therefore, $f(0)=0$. By Lemma \ref{equation nabla}, 
$$\Phi'^*(f^*(x \nabla y))=
f^*\left(\Phi'^*\left(x~\nabla~y \right)\right) =\Phi'^*\left( f^*(x)~ \nabla~  f^*(y) \right).$$
Since $\Phi'$ and $\Phi'^*$ are injective, from the last identity it implies that $f^*(x\nabla y)=f^*(x)\nabla f^*(y)$.
Therefore, $f=f_{!_\beta}$ for some $\beta\in (0,1)$; and hence, 
 $OM_+^1(\sigma_\Phi)\ni f_{!_\beta}$. 
By Proposition \ref{harmonic1}, $\sigma_\Phi =!_\alpha$ for some $\alpha\in(0,1)$.

\end{proof}

Notice that for any $A,B>0$,
$$f(A \sigma_{\Phi} B) =f(A) \sigma_{\Phi} f(B)$$
is equivalent to 
$$f^*(A \sigma_{\Phi^*} B) =f^*(A) \sigma_{\Phi^*} f^*(B).$$
Then, from Theorem \ref{equation main theorem} and  the duality between the weighted harmonic and weighted arithmetic means, we have the following. 

\begin{corollary}\label{equation:arithmetic}
Let $f, \Phi \in OM_+^1\backslash\{1,t\}$ with $\Phi^*(0) = 0$. 
If $f$ and $\Phi$ are non-trivial, then, 
the following statements are equivalent:
\begin{itemize}
\item[(I)] $f(A \sigma_{\Phi} B) =f(A) \sigma_{\Phi} f(B)$ for all $A,B>0$; 
\item[(II)] $\Phi=f_{\nabla_\alpha}$  and $f=f_{\nabla_\beta}$ for some $\alpha, \beta \in (0,1)$.
\end{itemize}
\end{corollary}

\begin{example}
Let $\mathrm{ALG}_p$ be the representative function 
for the power difference mean, which is defined by 
$$
\mathrm{ALG}_p(t) = 
\left\{\begin{array}{ll}
\dfrac{p-1}{p}\cdot \dfrac{1-t^p}{1-t^{p-1}}&t\not= 1 ;\\
1&t=1.\\
\end{array}
\right.
$$
For $p\in (-1,2)$, it is easy to see that either $ALG_p(0)=0$ or $ALG_p^*(0)=0$. 
Therefore, from Theorem \ref{equation main theorem} and Corollary \ref{equation:arithmetic},  
the equation (\ref{equation1}) implies that $f$ is trivial. 
\end{example}


For an arbitrary operator mean $\sigma$, let 
$OM_+^1(\sigma)_0:=\{f\in OM_+^1(\sigma)~|~f(0)=0\}.$ 
Then, it is clear that  
$$\{f\in OM_+^1~|~ f(0)=0, f(A \sigma B) =f(A) \sigma f(B) \text{ for all } A,B>0\}\subset 
OM_+^1(\sigma)_0.$$

In Theorem \ref{prp:operator inequality} 
(resp. Proposition \ref{prp:characterization of trivial}), 
we show that 
for a non-trivial operator mean $\sigma_\Phi$ such that $\Phi^*(0)=0$ (resp. $\Phi(0)+\Phi'(0)>0$), 
the class $OM_+^1(\sigma_\Phi)_0$ is trivial. 
Thus, the following is obtained. 

\begin{corollary}\label{33331}
Let $\Phi\in OM_+^1$. If either $\Phi^*(0)=0$ or $\Phi(0)+\Phi'(0)>0$, then 
a function $f\in OM_+^1$ satisfying $f(0)=0$ and equation (\ref{equation1})
is trivial. 
\end{corollary}

\begin{example}
The Petz-Hasegawa function $PH_p$ satisfies the following: 
$$PH_p(0)=0, \quad  PH_p^*(0)=0 \quad \text{ for }\quad  p\in [-1,0]\cup[1,2] $$ 
and
$$PH_p(0)=p(1-p), \quad  PH_p^*(0)=0  \quad \text{ for }\quad p\in (0,1). $$ 
Therefore, it follows from the above argument that a function $f\in OM_+^1$ satisfying $f(0)=0$ and equation (\ref{equation1})
is trivial. 

In addition, we can remove the condition $f(0)=0$ for the operator mean $\sigma_{PH_p}$.
Indeed, for $p\in [-1,0]\cup[1,2] $, 
the function $f\in OM_+^1$ satisfying (\ref{equation1})
is trivial by Theorem \ref{equation main theorem} and  Corollary \ref{equation:arithmetic}. 
Let us consider $p\in (0,1)$. Since 
$${{{d\over {dp}}PH_p}\over {PH_p}}={d\over {dp}}\log PH_p=
{1\over p}-{1\over {1-p}} +(\log t)\left( {1\over {t^{1-p}-1}}- {1\over {t^{p}-1}} \right)\ge 0$$
for $p\in (0,{1\over 2}]$, the function $p\mapsto PH_p$ is monotone increasing on $(0,{1\over 2}]$. 
Considering $PH_p=PH_{1-p}$, we have 
$$PH_0(t) =1\lambda t \le PH_p(t) \le PH_{1\over 2}(t) \lneq {{1+t}\over 2}.$$ 
This implies that
$\sigma_{PH_p}\not=\nabla_\alpha$ for all $\alpha\in [0,1]$. By Corollary \ref{equation:arithmetic}, 
$f$ satisfying (\ref{equation1}) must be trivial.  
\end{example}
\begin{example}
In Section \ref{section name subpreserving}, we described   
the Stolarsky mean $S_\alpha$ 
and showed that $OM_+^1(\sigma_{S_\alpha})_0=\{1,t\}$. 
Thus, for $S_\alpha$, a function $f\in OM_+^1$ satisfying $f(0)=0$ and (\ref{equation1}) must be trivial. 
Furthermore, 
since $S_\alpha(1,t)$ is not the arithmetic mean of $\alpha \in [-2,2)$, 
if $f$ preserves $S_\alpha$, then 
$f$ is trivial by Corollary \ref{33331}.
\end{example}

\section{Weighted power means}

For $t>0$, we consider  
\begin{equation*}
\Phi_{(r, \alpha)}(t):=
\begin{cases}
(\alpha t^r +(1-\alpha))^{1\over r}, & \text{if } r\in [-1,0)\cup(0,1];  \\
t^\alpha, & \text{if }  r=0.
\end{cases}
\end{equation*}
It is well-known that $\Phi_{(r, \alpha)}$ is in $OM_+^1$ for $r\in [-1,1]$ and $\alpha\in [0,1]$.
In addition, 
$$ \Phi_{(-1, \alpha)}=f_{!_\alpha},\quad  \Phi_{(0, \alpha)}=f_{\#_\alpha}, \quad 
\Phi_{(1, \alpha)}=f_{\nabla_\alpha}.$$ 
The mean $\sigma_{\Phi_{(r, \alpha)}}$ is called the weighted power mean. 

Although the weighted power mean $\Phi_{(r,\alpha)}$ does not satisfy 
the conditions in Theorem \ref{equation main theorem} and Corollary \ref{equation:arithmetic}
for some $r\in [-1,1]$, 
we establish some characterizations of functions $f$ preserving  $\Phi_{(r,\alpha)}$. 

Let us consider the following  operator mean $\sigma$:
\begin{equation}\label{integral mean}
t \sigma s = g(\alpha g^{-1}(t) + (1-\alpha) g^{-1}(s)), \quad s, t>0,\ \alpha \in (0,1),
\end{equation}
where $g$ is a bijective function from an open interval $I$ onto $(0,\infty)$. Some important means described by (\ref{integral mean}) are available in \cite{HLP}. For example, the weighted geometric mean is defined as $t \#_{1-\alpha} s = \exp(\alpha \log t + (1-\alpha)\log s)$ when  $g(t)=e^t$  $(t\in(-\infty,\infty))$, while the weighted power mean $ \Phi_{(r, \alpha)}$ is defined by (\ref{integral mean}) with $g(t)=t^{1\over a}$ $(a\in [-1,0)\cup (0,1])$.

\begin{proposition}\label{basic prop}
Let $f$ be a positive function on $(0,\infty)$ with $f(1)=1$. 
If an operator mean $\sigma$ satisfies (\ref{integral mean}), then the following statements are equivalent:
\begin{itemize}
\item[(I)] $f(t\sigma s)= f(t)\sigma f(s)$ for all $t,s >0$;
\item[(II)] $f(t)=g(\beta g^{-1}(t) + (1-\beta)g^{-1}(1)) $ for some $\beta\in \mathbb{R}$ and for all $t>0$. 
\end{itemize}
\end{proposition}
\begin{proof}(I) $\Rightarrow$ (II). 
Since $g(x) \sigma g(y) = g(\alpha x + (1-\alpha)y)$ for any $x,y\in I$, we have 
\begin{align*}
f(g(\alpha x + (1-\alpha)y))
&=f(g(x)\sigma g(y))\\
&=f(g(x))\sigma f(g(y))\\
&=g(\alpha (g^{-1}\circ f\circ g)(x)+ (1-\alpha) (g^{-1}\circ f\circ g)(y)   ). 
\end{align*}
Consequently, $g^{-1}\circ f\circ g $ is linear. Therefore, there exist $\beta$ and $\gamma$  in $\mathbb{R}$ 
such that $(g^{-1}\circ f\circ g)(t)=\beta t +\gamma$. For $t=g^{-1}(1)$, we obtain 
$$g^{-1}\circ f\circ g(t)=
g^{-1}\circ f(1)=
g^{-1}(1)=
\beta g^{-1}(1) +\gamma.$$
Thus, $\gamma=(1-\beta)g^{-1}(1)$, which implies the desired result. 

(II) $\Rightarrow$ (I). For $t,s>0$, 
\begin{align*}
f(t)\sigma f(s) &= g(\beta g^{-1}(t) +(1-\beta)g^{-1}(1)) \sigma g(\beta g^{-1}(s) + (1-\beta)g^{-1}(1))  \\
&= g( \alpha (\beta g^{-1}(t)+ (1-\beta)g^{-1}(1) ) +  (1-\alpha ) (\beta g^{-1}(s)+ (1-\beta)g^{-1}(1) )   ) \\
&= g(\beta ( \alpha g^{-1}(t) + (1-\alpha ) g^{-1}(s) ) + (1-\beta)g^{-1}(1) )  \\
&= g(\beta g^{-1}(t \sigma s) + (1-\beta)g^{-1}(1) ) \\
&= f(t\sigma s).
\end{align*}
\end{proof}

The following corollary is an immediate consequence of Proposition \ref{basic prop} for $g(t)=e^t$, $t\in (-\infty,\infty)$.

\begin{corollary}\label{formula equation r=0}Let $f$ be in $OM_+^1$ and let 
$t\sigma s = t^\alpha s^{1-\alpha}$. Then, 
the following statements are equivalent:
\begin{itemize}
\item[(I)] $f(t \sigma s) =f(t) \sigma f(s) $ for all $t,s>0$ ;
\item[(II)] $ f(t) =  t^\beta$
 for all $t>0$ and for some 
$\beta\in[0,1]$.
\end{itemize}
\end{corollary}


\begin{corollary}\label{formula equation r not 0}
Let $f$ be in $OM_+^1$ and let 
$t\sigma s= (\alpha t^a +(1-\alpha) s^a)^{1\over a}
 \ \ (a\not=0, -1\le a \le 1) $. Then, 
the following statements are equivalent:
\begin{itemize}
\item[(I)] $f(t \sigma s) =f(t) \sigma f(s) $ for all $t,s>0$ ;
\item[(II)] $ f(t) = ( \beta t^a  + (1-\beta) )^{1\over a} $ for all $t>0$ and for some 
$\beta \in [0,1]$.
\end{itemize}
\end{corollary}
\begin{proof} The implication (II) $\Rightarrow$ (I) follows from Proposition \ref{basic prop}. We show
the implication (I) $\Rightarrow$ (II). Since $f\in OM_+^1$ if and only if $f^*\in OM_+^1$, 
it is sufficient to prove  Corollary for the case $a\in (0,1]$. By Proposition \ref{basic prop}, there exists $\beta\in {\mathbb R}$ such that 
$$ f(t) = ( \beta t^a  + (1-\beta) )^{1\over a} .$$  
Since $f$ is monotone increasing, we have $\beta\ge 0$. Now, we assume that
$\beta>1$.  Then, 
$$\arg\left( \beta (r e^{i\theta})^a + 1-\beta \right) > a \theta$$
for a sufficiently small $r>0$. It follows that there exists a complex number $z$ such that 
$0\le \arg z \le \pi$ and 
 $\arg f(z)>\pi$, which contradicts $f\in OM_+^1$. Therefore, we have $\beta\in [0,1]$.

\end{proof}


The matrix generalization of Corollary \ref{formula equation r not 0} is as follows.
\begin{proposition}\label{main theorem1}
Let $r\in (-1,1)$, $\alpha \in (0,1)$ 
and let $f \in OM_+^1$. If $\Phi=\Phi_{(r,\alpha)}$, then the following statements are equivalent: 
\begin{itemize}
\item[(I)] $f(A \sigma_{\Phi} B) = f(A) \sigma_{\Phi} f(B)$ for all $A,B>0$;
\item[(II)] $f$ is trivial.
\end{itemize}
\end{proposition}


\begin{proof}
(I) $\Rightarrow$ (II). 
It is clear that (I) is equivalent to 
$$f^*(A\sigma_{\Phi^*}B)=f^*(A)\sigma_{\Phi^*}f^*(B).$$
On account of the relation $\Phi^* =\Phi_{(-r,\alpha)}$, 
it is sufficient to show Proposition for the case $r\in (-1,0]$. 

First, we consider the case $r\in (-1,0)$. By Corollary \ref{formula equation r not 0}, the function $f$ is written as $f = \Phi_{(r,\beta)}$ for some $\beta\in [0,1]$. Now, assume on the contrary that $f$ is non-trivial, i.e., $\beta\in (0,1)$. Then, it follows 
from $\Phi(0)=0$ and Theorem \ref{equation main theorem}  
that $f=\Phi_{(r,\beta)}=f_{!_a}$ for some $a\in (0,1)$. By a simple calculation, we have 
$$\left\{\Phi_{(r,\beta)}\left( {1\over t}\right)\right\}^r = 
\beta t^{-r}+(1-\beta )=(a t+(1-a))^{-r}
=
\left\{f_{!_a}\left( {1\over t}\right)\right\}^r,$$
and
$$\beta={d\over {dt}}\Phi_{(r,\beta)}(t)~\Big|_{t=1}
={d\over {dt}}(1~{!_a}~t)~\Big|_{t=1}=a.
$$
This result is a contradiction to the strict concavity of the function $t\mapsto t^{-r}$. Therefore, $f$ is trivial.

When $r=0$, by Theorem \ref{equation main theorem} and Corollary \ref{formula equation r=0}, 
the fact that 
$\Phi$ is not the weighted harmonic mean implies that $f$ is trivial. 

The implication (II) $\Rightarrow$ (I) is obvious.
\end{proof}
When $r\in \{-1,1\}$, a function $f$ satisfying 
$$f(A \sigma_{\Phi_{(r,\alpha)}} B) = f(A) \sigma_{\Phi_{(r,\alpha)}} f(B)\quad (A,B>0)$$
is not always trivial. 

In conclusion, we state the following proposition that follows from Theorem \ref{equation main theorem} and Corollary \ref{equation:arithmetic}.  
\begin{proposition}
Let $r\in \{-1,1\}$, $\alpha \in (0,1)$ 
and let $f \in OM_+^1$.
If $\Phi=\Phi_{(r,\alpha)}$, the following are equivalent: 
\begin{itemize}
\item[(I)] $f(A \sigma_{\Phi} B) = f(A) \sigma_{\Phi} f(B)$ for all $A,B>0$;
\item[(II)] $f=\Phi_{(r,\beta)}$ for some $\beta\in[0,1]$. 
\end{itemize}
\end{proposition}

\leftline{{\it Acknowledgement.}}

The research of the second author is partially supported by JSPS KAKENHI Grant Number JP17K05285.



\begin{thebibliography}{99}


\bibitem{A} T. Ando, 
{\it{Topics on operator inequalities}}.
Hokkaido Univ. Lecture Note, 1978.

\bibitem{ando} T. Ando and F. Hiai. {\it Log majorization and complementary Golden-Thompson type inequalities.} Linear Algebra Appl. 197/198 (1994), 113--131. 


\bibitem{AH01} T. Ando and F. Hiai.  
{\it Operator log-convex functions and operator means}.
Math. Ann. 350(2011),  611--630.
\bibitem{hoa} T.H. Dinh, T.D. Dinh and B.K. Vo. {\it A new type of operator convexity}. Acta Mathematica Vietnamica.  43:4 (2018), 595--605
\bibitem{gn} M. Gaal and G. Nagy. {\it Preserver problems related to quasi-arithmetic means of invertible positive operators.} Integral Equations and Operator Theory. 90:7 (2018).

\bibitem{FJ92} J.I. Fujii,
{\it{Operator means and Range inclusion}}. Linear Algebra Appl.~170
 (1992), 137--146.

\bibitem{FN} J.I. Fujii and M. Nakamura, 
{\it{A characterization of the harmonic operator mean as an extension of Ando's theorem}}.
Sci. Math. Jpn. 63(2006),no.2,205--210. 


\bibitem{HLP} G.H. Hardy, J.E. Littlewood and G. P\'olya,
{\it{ Inequalities 2d ed.}}.
 Cambridge University Press, 1952. 

\bibitem{HP}
F. Hiai and D. Petz,
{\it{Introduction to matrix analysis and applications}}. 
Universitext. Springer, Cham;
Hindustan Book Agency, New Delhi, 2014.

\bibitem{KA} F. Kubo and T. Ando, 
{\it{ Means of positive linear operators}}.
 Math. Ann. ~246 (1980), 205--224. 
\bibitem{molnar1} L. Molnar. {\it Maps preserving the geometric mean of positive operators.} Proc. Amer. Math. Soc. 137 (2009), 1763--1770.
\bibitem{molnar2} L. Molnar. {\it Maps preserving general means of positive operators}. Electron. J. Linear Algebra. 22(2011), 864--874.
\bibitem{molnar3} L. Molnar. {\it Maps preserving the harmonic mean or the parallel sum of positive operators.} Linear Algebra Appl. 430 (2009), 3058--3065.
 
\bibitem{N}
Y. Nakamura,
{\it{Classes of operator monotone functions and Stieltjes functions}}. H. Dym, et al. (Eds.), The Gohberg Anniversary Collection, vol. II, 
Oper. Theory Adv. Appl., vol. 41, Birkh\"auser, 1989, pp. 395--404.


\bibitem{ST}K.B. Stolarsky, 
{\it{Generalizations of the logarithmic mean}}.
Math. Mag. 48(2)(1975), 87--92.



\end{thebibliography}
\end{document}